\documentclass[10pt,reqno,twoside]{amsart}
\usepackage{amssymb,amsmath,amsthm,soul,color,paralist}
\usepackage{t1enc}
\usepackage{comment}
\usepackage[cp1250]{inputenc}
\usepackage{a4,indentfirst,latexsym}
\usepackage{graphics}
\usepackage{mathrsfs}
\usepackage{cite,enumitem,graphicx}
\usepackage[colorlinks=true,urlcolor=blue,
citecolor=red,linkcolor=blue,linktocpage,pdfpagelabels,
bookmarksnumbered,bookmarksopen]{hyperref}
\usepackage[english]{babel}
\usepackage[left=2.50cm,right=2.50cm,top=2.72cm,bottom=2.72cm]{geometry}
\usepackage[metapost]{mfpic}
\usepackage[colorinlistoftodos]{todonotes}
\usepackage[normalem]{ulem}

\numberwithin{equation}{section}

\allowdisplaybreaks

\def\Xint#1{\mathchoice
  {\XXint\displaystyle\textstyle{#1}}%
  {\XXint\textstyle\scriptstyle{#1}}%
  {\XXint\scriptstyle\scriptscriptstyle{#1}}%
  {\XXint\scriptscriptstyle\scriptscriptstyle{#1}}%
  \!\int}
\def\XXint#1#2#3{{\setbox0=\hbox{$#1{#2#3}{\int}$}
  \vcenter{\hbox{$#2#3$}}\kern-.5\wd0}}
\def\-int{\Xint -}

\newcommand{\R}{\mathbb{R}}

\newcommand{\ri}{\rightarrow}

\DeclareMathOperator{\dive}{div}

\DeclareMathOperator{\e}{\varepsilon}

\newtheorem{lem}{Lemma}[section]
\newtheorem{thm}{Theorem}[section]

\newtheorem{remark}{Remark}[section]

\begin{document}
\title[On the Pohozaev identity for the fractional $p$-Laplacian operator in $\R^N$]
{On the Pohozaev identity for the fractional $p$-Laplacian operator in $\R^N$}

\author[V. Ambrosio]{Vincenzo Ambrosio}
\address{Vincenzo Ambrosio \hfill\break\indent
Dipartimento di Ingegneria Industriale e Scienze Matematiche \hfill\break\indent
Universit\`a Politecnica delle Marche\hfill\break\indent
Via Brecce Bianche, 12\hfill\break\indent
60131 Ancona (Italy)}
\email{v.ambrosio@staff.univpm.it}


\maketitle

\begin{abstract}
In this paper, we show the existence of a nontrivial weak solution for a nonlinear problem involving the fractional $p$-Laplacian operator and a Berestycki-Lions type nonlinearity. This solution satisfies a Pohozaev identity. Moreover, we prove that any sufficiently smooth solution fulfills the Pohozaev identity.
\end{abstract}

\section{Introduction}
Let $N\geq 2$, $s\in (0, 1)$ and $p\in (1, \frac{N}{s})$.
Let us consider the following nonlinear problem
\begin{equation}\label{P}
\left\{
\begin{array}{ll}
(-\Delta)^{s}_{p} u= g(u) \, \mbox{ in } \R^{N}, \\
u\in W^{s, p}(\R^{N}), 
\end{array}
\right.
\end{equation}
where $(-\Delta)^{s}_{p}$ is the fractional $p$-Laplacian operator (see \cite{AMRT, DTGCV, IMS, IN})
defined for all $u\in L^{p-1}_{s}$ by
\begin{align*}
(-\Delta)^{s}_{p}u(x)&=C_{N, s, p}\,\, P.V. \int_{\R^{N}} \frac{|u(x)-u(y)|^{p-2}(u(x)-u(y))}{|x-y|^{N+sp}}dy, \\
&=C_{N, s, p}\,\, \lim_{\e\ri 0^{+}} \int_{|y-x|>\e} \frac{|u(x)-u(y)|^{p-2}(u(x)-u(y))}{|x-y|^{N+sp}}dy, 
\end{align*}
provided the limit exists (here $P.V.$ stands for the Cauchy principal value), 
with
$$
C_{N, s, p}=\frac{\frac{sp}{2}(1-s)2^{2s-1}}{\pi^{\frac{N-1}{2}}} \frac{\Gamma(\frac{N+sp}{2})}{\Gamma(\frac{p+1}{2})\Gamma(2-s)},
$$
and
$$
L^{p-1}_{s}=\Bigl\{u:\R^{N}\ri \R \mbox{ measurable } : \int_{\R^{N}}\frac{|u(x)|^{p-1}}{1+|x|^{N+sp}}dx<\infty\Bigr\}.
$$
The nonlinearity $g:\R\ri \R$ is a continuous odd function that obeys the following Berestycki-Lions type conditions \cite{BL}:
\begin{compactenum}[$(g1)$]
\item $\displaystyle{-\infty<\liminf_{t\ri 0^{+}} \frac{g(t)}{t^{p-1}}\leq \limsup_{t\ri 0^{+}} \frac{g(t)}{t^{p-1}}=-m}$, with $m>0$,
\item $\displaystyle{-\infty\leq \limsup_{t\ri \infty} \frac{g(t)}{t^{p^{*}_{s}-1}}\leq 0}$, where $p^{*}_{s}=\frac{Np}{N-sp}$ is the fractional critical exponent,
\item there exists $\zeta>0$ such that $G(\zeta)>0$, where $G(t)=\int_{0}^{t}g(\tau)d\tau$.
\end{compactenum}
Our first result concerns the existence of a weak solution to \eqref{P} satisfying a Pohozaev type identity. More precisely, we prove the following result. 
\begin{thm}\label{thm1}
Let $N\geq 2$, $s\in (0, 1)$ and $p\in (1, \frac{N}{s})$. Assume that $(g1)$-$(g3)$ hold. Then there exists a nontrivial weak solution $u\in W^{s, p}(\R^{N})$ to \eqref{P} fulfilling the following Pohozaev identity: 
\begin{align}\label{POHOZAEV}
\frac{C_{N, s, p}(N-sp)}{2p}\iint_{\R^{2N}} \frac{|u(x)-u(y)|^{p}}{|x-y|^{N+sp}}dxdy-N\int_{\R^{N}} G(u)dx=0.
\end{align}
\end{thm}
We recall that the Pohozaev identity for the $p$-Laplacian operator $\Delta_{p}u=\dive(|\nabla u|^{p-2}\nabla u)$ has been extensively investigated in the literature; see for instance \cite{DGMS, GV, PS}.
On the other hand, in the fractional setting, if $s\in (0, 1)$ and $p=2$, then $(-\Delta)^{s}_{p}$ becomes the fractional Laplacian operator $(-\Delta)^{s}$, and the corresponding Pohozaev identity for weak solutions to \eqref{P} has been established in \cite[Proposition 4.1]{CW} (see also \cite[Theorem 3.5.1]{A} for more details and \cite[Theorem 1.1]{RS} for the case of bounded domains). More precisely,  inspired by \cite[Proposition 1]{BL}, the authors in \cite{CW} required that $g\in C^{1}(\R)$ and derived the Pohozaev identity for $(-\Delta)^{s}$ by employing the extension method \cite{CS}. They combined some regularity results for $(-\Delta)^{s}$ in $\R^{N}$ and for the operator $\dive(y^{1-2s} \nabla)$ in the upper half-space $\R^{N+1}_{+}=\{(x, y)\in \R^{N+1}: y>0 \}$, and performed integration by parts. 
When $p\neq 2$, we encounter some complications in adapting the approach in \cite{CW}. Although the extension method for $(-\Delta)^{s}_{p}$ has been recently explored in \cite[section 3]{DTGCV}, we do not have enough information about the regularity of the extension $U(x, y)$ of $u(x)$ in $\R^{N+1}_{+}$ whenever $p\neq 2$. Furthermore, in contrast to the linear case $p=2$, it remains unclear whether the following identity is valid:
$$
C\iint_{\R^{2N}} \frac{|u(x)-u(y)|^{p}}{|x-y|^{N+sp}}dxdy=\iint_{\R^{N+1}_{+}} y^{-1+p(1-s)} |\nabla U|^{p}dxdy,
$$
where $C>0$ is an appropriate constant. 
Therefore, we need to follow a different strategy to arrive at \eqref{POHOZAEV}. We stress that the main difficulty in reaching \eqref{POHOZAEV} consists in using $x\cdot \nabla u$, where $u$ is a weak solution to \eqref{P}, as a test function in the weak formulation of \eqref{P} and subsequently applying an integration by parts. Indeed, due to the fact that for $p\neq 2$ we have to handle the nonlinearity of the operator $(-\Delta)^{s}_{p}$ and its nonlocal character, it seems to be a challenging task to verify that $x\cdot \nabla u\in W^{s, p}(\R^{N})$. Moreover, it seems to be hard to accomplish an integration by parts formula for $(-\Delta)^{s}_{p}$. 
In this paper, by means of suitable variational methods for potential operators with covariance condition (see \cite{B, Chab}), we show the existence of a weak solution to \eqref{P} that satisfies \eqref{POHOZAEV}.
We recall that $u\in W^{s, p}(\R^{N})$ is a weak solution to \eqref{P} if  $\langle I'(u),\phi\rangle=0$ for all $\phi\in W^{s, p}(\R^{N})$, where $I:W^{s, p}(\R^{N})\ri \R$ is  the energy functional associated with \eqref{P}, namely, 
$$
I(u)=\frac{C_{N, s, p}}{2p}\iint_{\R^{2N}} \frac{|u(x)-u(y)|^{p}}{|x-y|^{N+sp}}dxdy-\int_{\R^{N}} G(u)dx.
$$
To implement the abstract results in \cite{B, Chab}, we work with $I$ on the radial subspace $W^{s, p}_{\rm rad}(\R^{N})$ and exploit two fundamental facts: the Schwarz symmetrization decreases the Gagliardo seminorm in $W^{s, p}(\R^{N})$ \cite[Theorem 9.2]{AL}, and $W^{s, p}_{\rm rad}(\R^{N})$ is compactly embedded into $L^{r}(\R^{N})$ for all $r\in (p, p^{*}_{s})$ \cite[Theorem II.1]{Lions}; see Theorem \ref{THMBL}.
We emphasize that Theorem \ref{thm1} guarantees the existence of a weak solution to \eqref{P} fulfilling \eqref{POHOZAEV}, but unlike the case $p=2$, it does not assert that every weak solution to \eqref{P} satisfies \eqref{POHOZAEV}.
However, this last statement is true for $C^{1, 1}$ functions that solve the equation in \eqref{P} pointwise.
More precisely, our second main result can be stated as follows. 
\begin{thm}\label{thm2}
Let $N\geq 2$, $s\in (0, 1)$ and $p\in (1, \frac{N}{s})$.  
When $p\in (1, 2)$, we assume that $s<\frac{2(p-1)}{p}$. Assume that $(g1)$-$(g3)$ hold.
Let $u\in W^{s, p}(\R^{N})\cap C^{1, 1}(\R^{N})$ be such that $(-\Delta)^{s}_{p}u=g(u)$ in $\R^{N}$. Then \eqref{POHOZAEV} is valid.
\end{thm}
The proof of Theorem \ref{thm2} goes as follows. First we prove an integral representation formula for $(-\Delta)^{s}_{p}v$, with $v\in C^{1, \gamma}(\R^{N})$ and convenient $\gamma\in (0, 1]$, in the spirit of \cite[Lemma 3.2]{DPV} (see Lemma \ref{RIF}). Second, motivated by \cite[Lemma 4.2]{DFW1}, we establish an integration by parts formula for $W^{s, p}\cap C^{1, 1}$ functions and vector fields of class $C^{0, 1}_{c}$ (see Lemma \ref{LEMMA}). 
Finally, we multiply the equation in \eqref{P} by $\varphi(\lambda x) x\cdot \nabla u$, where $\varphi\in C^{1}_{c}(\R^{N})$ is such that $\varphi=1$ in a neighborhood of $0$ and $\lambda>0$, integrate over $\R^N$, apply our integration by parts formula, and then take the limit as $\lambda\ri 0^{+}$ in the resulting relation to obtain the desired identity. 

An appropriate comment on the assumption $u\in C^{1, 1}$ is necessary. In view of $(g1)$ and $(g2)$, it is possible to check that every weak solution to \eqref{P} is H\"older continuous (see Remark \ref{REMAI}), but we do not know if this regularity is optimal; see \cite{BLS} for a more detailed discussion on this subject. Nevertheless, we suspect that one could achieve the Pohozaev identity \eqref{POHOZAEV} by approximating \eqref{P} with more regular problems depending on a small parameter $\e>0$. For instance, consider mixed local and nonlocal problems driven by $-\e \Delta_{p}+(-\Delta)^{s}_{p}$, as in \cite{DFM}, whose solutions are $C^{1, \alpha}_{\rm loc}$, and try to perform an integration by parts before passing to the limit as $\e\ri 0^{+}$.
A different approach to discover \eqref{POHOZAEV} would be to work with difference quotients of dilations, instead of 
utilizing directly the Euler vector field $x\cdot \nabla u$. However, in this paper we assume that $u\in C^{1, 1}$ and provide a simple and self-contained proof of \eqref{POHOZAEV} that we believe to be useful for future works.

To the best of our knowledge, this is the first time that a Pohozaev identity for the fractional $p$-Laplacian operator $(-\Delta)^{s}_{p}$ in $\R^{N}$ has been obtained in the literature. \\
The paper is organized as follows. In Section $2$ we give the proof of Theorem \ref{thm1}. Section $3$ is devoted to the proof of Theorem \ref{thm2}.

\subsection*{Notations:}
Henceforth, $\|\cdot\|_{L^{p}(\R^{N})}$ denotes the $L^{p}(\R^N)$-norm with $p\in [1, \infty]$. For $s\in (0, 1)$ and $p\in [1, \infty)$, the fractional Sobolev space 
$W^{s, p}(\R^{N})$ is the set of functions $u\in L^{p}(\R^{N})$ such that $[u]_{s, p}^{p}=\iint_{\R^{2N}} \frac{|u(x)-u(y)|^{p}}{|x-y|^{N+sp}}dxdy<\infty$. The space $W^{s, p}(\R^{N})$ is a Banach space with the norm 
$$
\|u\|_{W^{s, p}(\R^{N})}=(\|u\|_{L^{p}(\R^{N})}^p+[u]^{p}_{s, p})^{\frac{1}{p}}.
$$
With $C^{k, \gamma}(\R^{N})$, where $k\in \mathbb{N}\cup \{0\}$ and $\gamma\in (0, 1]$, we denote the set of functions $u\in C^{k}(\R^N)$ whose partial derivatives up to order $k$ are bounded and such that 
\begin{align*}
 \|u\|_{C^{k, \gamma}(\R^{N})}
 =\sum_{|\beta|\leq k}\|D^{\beta}u\|_{L^{\infty}(\R^{N})}+\sum_{|\beta|=k} \, \sup_{\substack{x, y\in \mathbb{R}^{N} \\ x\neq y}} \frac{|D^{\beta}u(x)-D^{\beta}u(y)|}{|x-y|^{\gamma}}<\infty,
\end{align*} 
where we have used the multi-index notation, that is, if $u\in C^{k}(\R^{N})$ and $\beta=(\beta_{1}, \dots, \beta_{N})\in (\mathbb{N}\cup\{0\})^{N}$ is a multi-index of length $|\beta|=\beta_{1}+\dots+\beta_{N}\leq k$, then $D^{\beta} u=\frac{\partial^{|\beta|}u}{{\partial x_{1}^{\beta_{1}}}\dots {\partial x_{N}^{\beta_{N}}}}$ (with the convention that $D^{0}u=u$).

\section{Proof of Theorem \ref{thm1}}
In this section, we examine the existence of a weak solution to \eqref{P} by means of two abstract results found in \cite[chapter 4]{Chab} and motivated by \cite{B}. For the reader's convenience, we state them below. We first introduce some definitions.

Let $X$ be a reflexive Banach space equipped with norm $\|\cdot\|$ and denote by $X^{*}$ its dual. 
A map $A:X\ri X^{*}$ is said to be a potential operator with a potential $a:X\ri \R$, if $a$ is Gateaux differentiable and
$$
\lim_{t\ri 0} \frac{a(u+tv)-a(u)}{t}=\langle A(u), v\rangle \quad \mbox{ for all } u, v\in X.
$$
For a potential, we always assume that $a(0)=0$. 
Let $A:X\ri X^{*}$ and $B:X\ri X^{*}$ two potential operators with potentials $a:X\ri \R$ and $b:X\ri \R$, respectively.
Suppose that $b$ is defined on a set $D(b)\subset X$.
For all $\sigma>0$, let $T_{\sigma}: X\ri X$ be a linear map such that $T_{\sigma_1}\circ T_{\sigma_{2}}=T_{\sigma_{1}\sigma_{2}}$ for all $\sigma_{1}, \sigma_{2}>0$, and $T_{1}=id$. Let $E$ be a subspace of $X$ such that $T_{\sigma}(E)\subset E$ for each $\sigma>0$. We assume the following covariance conditions on the potentials $a$ and $b$:
\begin{compactenum}[$(C1)$]
\item $a: X\ri \R$ and there exists $q_{1}\in \R$ such that $a(T_{\sigma}u)=\sigma^{q_{1}}a(u)$ for each $\sigma>0$ and $u\in D(b)$.
\item $b: D(b)\ri \R$ with $E\subset D(b)$, $D(b)+E\subset D(b)$, $T_{\sigma}(D(b))\subset D(b)$, and there exists $q_{2}\in \R$ such that $b(T_{\sigma}u)=\sigma^{q_{2}}b(u)$ for each $\sigma>0$ and $u\in D(b)$.
\end{compactenum}
Assuming $\{u\in D(b): b(u)=1\}\neq \emptyset$, we consider
\begin{align}\label{J1}
J=\inf\{a(u): u\in D(b), \, b(u)=1 \}.
\end{align}
\begin{thm}\label{THM4.1.1}\cite[Theorem 4.1.1-(i)]{Chab}
Suppose that $A$ is a potential operator with a potential $a$. Moreover, suppose that the problem \eqref{J1} has a solution $u$ and that $b$ has a linear continuous Gateaux derivative $\langle B(u), v\rangle$ at all directions $v\in E$. Then
$$
\langle A(u), v\rangle=J \frac{q_{1}}{q_{2}}\langle B(u), v\rangle \quad \mbox{ for all } v\in E.
$$
The functionals $a$ and $b$ have linear continuous Gateaux derivatives at all directions $T_{\sigma}u$, $\sigma>0$, satisfying 
$$
\langle A(T_{\sigma}u), v\rangle=\frac{q_{1}}{q_{2}}J \sigma^{q_{1}-q_{2}}\langle B(T_{\sigma}u), v\rangle \quad \mbox{ for all } v\in E \mbox{ and } \sigma>0.
$$
Furthermore, if $\frac{q_{1}}{q_{2}}J>0$ and $q_{1}-q_{2}\neq 0$, then the scaled minimizer $\bar{u}=T_{\bar{\sigma}}u$, with $\bar{\sigma}=\left(\frac{q_{1}}{q_{2}}J\right)^{\frac{1}{q_{2}-q_{1}}}$, satisfies
$$
\langle A(\bar{u}), v\rangle=\langle B(\bar{u}), v\rangle \quad \mbox{ for all } v\in E.
$$
\end{thm}
Next we assume that $a$ and $b$ satisfy $(C1)$ and $(C2)$ with $D(b)=X$ and $q_{1}\neq q_{2}$.
\begin{thm}\label{THM4.1.2}\cite[Theorem 4.1.2]{Chab}
Suppose that there exists a Banach subspace $X_{1}\subset X$ equipped with a norm from $X$ and a mapping $T:X\ri X_{1}$ such that $a(Tu)\leq a(u)$ and $b(Tu)=b(u)$ for each $u\in X$ and let $b(u_{0})>0$ for some $u_{0}\in X$. Moreover, we assume that $a$ is weakly lower semicontinuous and that $b$ admits a decomposition $b(u)=b_{1}(u)-b_{2}(u)$ on $X$, with $b_{1}$ and $b_{2}$ nonnegative on $X$, where $b_{1}$ restricted on $X_{1}$ is weakly continuous, $b_{2}$ is weakly lower semicontinuous and 
\begin{compactenum}[$(i)$]
\item there exist constants $\e\in (0, 1)$, $C>0$ and $\alpha>0$ such that 
$$
b_{1}(u)\leq \e b_{2}(u)+Ca(u)^{\alpha} \quad \mbox{ for all } u\in X,
$$
\item there exist constants $\beta>0$, $C'>0$ and $\gamma>0$ such that 
$$
\|u\|^{\beta}\leq C'(b_{2}(u)+a(u)^{\gamma}) \quad \mbox{ for all } u\in X.
$$
\end{compactenum}
Then problem \eqref{J1} has a nontrivial solution.
\end{thm}

\noindent
Now we prove an existence result for \eqref{P}.
\begin{thm}\label{THMBL}
Let $N\geq 2$, $s\in (0, 1)$ and $p\in (1, \frac{N}{s})$.
Then, \eqref{P} admits a nontrivial weak solution.
\end{thm}
\begin{proof}
Set $X=W^{s,p}(\R^{N})$ and $X_{1}=W^{s, p}_{\rm rad}(\R^{N})=\{u\in W^{s, p}(\R^{N}): u \mbox{ is radially symmetric } \}$. It is easy to check that $X$ is a reflexive Banach space. Indeed, to confirm the reflexivity of $X$, let $Y=L^{p}(\R^{N})\times L^{p}(\R^{2N})$ be endowed with the norm 
$$
\|(u, v)\|_{Y}=\left(\|u\|^{p}_{L^{p}(\R^{N})}+\|v\|^{p}_{L^{p}(\R^{2N})} \right)^{\frac{1}{p}} \quad \mbox{ for all } (u, v)\in Y,
$$
and introduce the linear isometry $\mathcal{L} :W^{s, p}(\R^{N})\ri Y$ defined as 
$$
\mathcal{L}(u)=\left(u, \frac{u(x)-u(y)}{|x-y|^{\frac{N+sp}{p}}} \right).
$$ 
Obviously, the product space $Y$ is a reflexive Banach space.
As $W^{s, p}(\R^{N})$ is a Banach space, $\mathcal{L}(W^{s, p}(\R^{N}))$ is a closed subspace of $Y$. It follows that $\mathcal{L}(W^{s, p}(\R^{N}))$ is reflexive. As a result, $W^{s, p}(\R^{N})$ is reflexive (this remains true for all $s\in (0, 1)$ and $p\in (1, \infty)$).

Set
$$
\|u\|=\left(\frac{C_{N, s, p}}{2} [u]_{s, p}^{p}+\|u\|^{p}_{L^{p}(\R^{N})}\right)^{\frac{1}{p}} \quad \mbox{ for all } u\in X.
$$
Let $T_{\sigma}: X\ri X$ be defined by $T_{\sigma}u(x)=u(\frac{x}{\sigma})$, $\sigma>0$, and $T:X\ri X_{1}$ given by $Tu(x)=u^{*}(|x|)$ (Schwarz symmetrization). 
Modifying $g$ as in \cite{BL}, we may assume that $g$ satisfies the stronger condition
\begin{align}\label{G2'}
\lim_{|t|\ri \infty} \frac{|g(t)|}{|t|^{p^{*}_{s}-1}}=0, 
\end{align}
instead of $(g2)$.
Put 
$$
a(u)=\frac{C_{N, s, p}}{2p} [u]^{p}_{s, p} \,\,\mbox{ and } \,\, b(u)=\int_{\R^{N}} G(u)dx.
$$ 
Note that $a, b\in C^{1}(X, \R)$, $a(u)\geq 0$, $a(T_{\sigma}u)=\sigma^{N-sp} a(u)$ and $b(T_{\sigma}u)=\sigma^{N}b(u)$ for each $u\in X$ and $\sigma>0$. We also have
$a(Tu)\leq a(u)$ (see \cite[Theorem 9.2]{AL}) and $b(Tu)=b(u)$ (see \cite[appendix A.III]{BL}) for all $u\in X$. Furthermore, $a$ is weakly lower semicontinuous on $X$.

Let us now show that the set
$\{u\in X: b(u)=1\}$ is nonempty.
For $R>1$, define
\begin{equation*}
w_{R}(x)=\left\{
\begin{array}{ll}
\zeta & \mbox{ for } |x|\leq R, \\
\zeta(R+1-|x|) & \mbox{ for } R\leq |x|\leq R+1, \\
0 & \mbox{ for } |x|\geq R+1,
\end{array}
\right.
\end{equation*}
where $\zeta>0$ is given in $(g3)$.
It is evident that $w_{R}\in X$ and 
$$
b(w_{R})\geq G(\zeta)|B_{R}(0)|-|B_{R+1}(0)\setminus B_{R}(0)| \max_{t\in [0, \zeta]} |G(t)|.
$$
Then there exist two constants $C_{1}, C_{2}>0$ such that
$$
b(w_{R})\geq C_{1}R^{N}-C_{2}R^{N-1},
$$
and so $b(w_{R})>0$ for $R>0$ large enough. 

Next we consider the following constrained minimization problem:
\begin{align}\label{J}
J=\inf\{a(u): u\in X, \, b(u)=1\}.
\end{align}
Let $g_{1}(t)=(g(t)+mt^{p-1})^{+}$ and $g_{2}(t)=g_{1}(t)-g(t)$ for all $t\geq 0$. We extend $g_{1}$ and $g_{2}$ as odd functions for $t\leq 0$. Set $b_{1}(u)=\int_{\R^{N}} G_{1}(u)dx$ and $b_{2}(u)=\int_{\R^{N}}G_{2}(u)dx$, where $G_{i}(t)=\int_{0}^{t}g_{i}(\tau)d\tau$ for $i=1, 2$. Therefore, $b(u)=b_{1}(u)-b_{2}(u)$ for all $u\in X$. Clearly, $b_{1}(u), b_{2}(u)\geq 0$ for all $u\in X$, and $b_{2}$ is weakly lower semicontinuous on $X$. Let us prove that $b_{1}$ is weakly continuous on $X_{1}$. Let $(u_{n})\subset X_{1}$ be such that $u_{n}\rightharpoonup u$ in $X_{1}$ and we show that $b_{1}(u_{n})\ri b_{1}(u)$. By \cite[Theorem II.1]{Lions}, we know that $X_{1}$ is compactly embedded into $L^{r}(\R^{N})$ for all $r\in (p, p^{*}_{s})$. We note that $G_{1}\in C(\R)$, $\lim_{|t|\ri 0} \frac{G_{1}(t)}{|t|^{p}}=0=\lim_{|t|\ri \infty} \frac{G_{1}(t)}{|t|^{p^{*}_{s}}}$ (by the definition of $g_{1}$ and the assumptions $(g1)$ and \eqref{G2'}), $\sup_{n\in \mathbb{N}}\|u_{n}\|\leq C$ and $G(u_{n})\ri G(u)$ almost everywhere in $\R^{N}$. Applying \cite[Lemma 2.4]{CW}, we deduce that $\|G(u_{n})-G(u)\|_{L^{1}(\R^{N})}\ri 0$ and thus $b_{1}(u_{n})\ri b_{1}(u)$, as desired.

Hereafter, we verify that the conditions $(i)$ and $(ii)$ of Theorem \ref{THM4.1.2} hold. 
Using $(g1)$ and $(g2)$, we see that for all $\e>0$ there exists a constant $C_{\e}>0$ such that 
$$
G_{1}(t)\leq \e G_{2}(t)+C_{\e}|t|^{p^{*}_{s}} \quad \mbox{ for all } t\in \R,
$$
and thanks to the fractional Sobolev inequality (see \cite[Theorem 6.5]{DPV}) we arrive at 
$$
b_{1}(u)\leq \e b_{2}(u)+C_{\e} C_{*} a(u)^{\frac{p^{*}_{s}}{p}} \quad \mbox{ for all } u\in X,
$$
for some constant $C_{*}>0$.
On the other hand, because $g_{2}(t)\geq m t^{p-1}$ for all $t\geq 0$, we have
$$
b_{2}(u)\geq \frac{m}{p}\|u\|^{p}_{L^{p}(\R^{N})} \quad \mbox{ for all } u\in X,
$$
which yields
$$
\|u\|^{p}\leq \frac{p}{m} b_{2}(u)+a(u) \quad \mbox{ for all } u\in X.
$$
Then, by virtue of Theorem \ref{THM4.1.2}, we infer that problem \eqref{J} has a nontrivial solution $u$. Observing that $\frac{(N-sp)}{N}J>0$ and $(N-sp)-N=-sp\neq 0$, we can invoke Theorem \ref{THM4.1.1} to discover that $\bar{u}=T_{\bar{\sigma}}u$, where $\bar{\sigma}=(\frac{N-sp}{N}J)^{\frac{1}{sp}}$, satisfies
\begin{align*}
\langle a'(\bar{u}), v\rangle=\langle b'(\bar{u}), v\rangle \quad \mbox{ for all } v\in X,
\end{align*}
namely, $\bar{u}$ is a weak solution to \eqref{P}. 
\end{proof}

As byproduct of Theorem \ref{THMBL}, we obtain Theorem \ref{thm1}.
\begin{proof}[Proof of Theorem \ref{thm1}]
Let $\bar{u}$ be the solution found in Theorem \ref{THMBL}.
From $a(\bar{u})=\bar{\sigma}^{N-sp}J$, $J>0$, $b(\bar{u})=\bar{\sigma}^{N}$ and $\bar{\sigma}=(\frac{N-sp}{N}J)^{\frac{1}{sp}}$, we can see that 
\begin{align*}
a(\bar{u})=\frac{N}{N-sp}b(\bar{u}),
\end{align*}
that is, \eqref{POHOZAEV} is valid.
\end{proof}

\begin{remark}
The proof of Theorem \ref{thm1} still works, with the appropriate modifications, in the case $s=1$.
\end{remark}

\begin{remark}\label{REMAI}
Reasoning as in \cite[Lemma 3.18]{AIDCDS} (see also \cite[Theorem 1.1]{ANA}), we can show that if $u$ is a nontrivial nonnegative weak solution to \eqref{P}, then $u\in C^{0, \alpha}(\R^{N})\cap L^{r}(\R^{N})$ for all $r\in [p, \infty]$ and for some $\alpha\in (0, 1)$ (note that this is true even if $u$ is a sign-changing solution). From the strong maximum principle \cite[Theorem 1.4]{DPQJDE}, $u>0$ in $\R^{N}$. As $(-\Delta)^{s}_{p}u+m u^{p-1}\leq g_{1}(u)$ in $\R^{N}$, where $g_{1}(t)=o(t^{p-1})$ as $t\ri 0^{+}$, we can proceed as in \cite[Corollary 2.1]{ANA} 
to conclude that $0<u(x)\leq C|x|^{-\frac{N+sp}{p-1}}$ for all $|x|$ large enough. If in addition $u\in C^{1, 1}_{\rm loc}(\R^{N})$ and $g\in C^{1}(\R)$, then $u$ must be radially symmetric and monotone decreasing about some point in $\R^N$, according to  \cite[Theorem 5]{CL}.
\end{remark}

\section{Proof of Theorem \ref{thm2}}

We start by giving a useful integral representation formula for $(-\Delta)^{s}_{p}$ along functions of class $C^{1, \gamma}(\R^{N})$ with suitable $\gamma\in (0, 1]$. This formula is motivated by \cite[Lemma 3.2]{DPV} in which the authors dealt with the case $p=2$ and considered functions in the Schwartz space of rapidly decaying functions. 
\begin{lem}\label{RIF}
Let $N\in \mathbb{N}$, $s\in (0, 1)$ and $p\in (1, \infty)$. 
When $p\in (1, 2)$, we assume that $s<\frac{2(p-1)}{p}$. 
Let $\gamma\in (0, 1]$ be such that $\gamma>1-p(1-s)$ if $p\geq 2$, whereas $\gamma>\frac{1-p(1-s)}{p-1}$ if $p\in (1, 2)$.
Then, for all $u\in C^{1, \gamma}(\R^{N})$, it holds, for all $x\in \R^N$,
\begin{align*}
(-\Delta)^{s}_{p}u(x)=\frac{C_{N, s, p}}{2} \int_{\R^{N}} \frac{|u(x)-u(x+z)|^{p-2}(u(x)-u(x+z))+|u(x)-u(x-z)|^{p-2}(u(x)-u(x-z))}{|z|^{N+sp}} dz.
\end{align*}
\end{lem}
\begin{proof}
Let $u\in C^{1, \gamma}(\R^{N})$ and fix $x\in \R^{N}$. Using the change of variables theorem and the symmetry of the kernel $\frac{1}{|z|^{N+sp}}$, we have 
\begin{align*}
&P.V.\int_{\R^{N}} \frac{|u(x)-u(y)|^{p-2}(u(x)-u(y))}{|x-y|^{N+sp}} dy\\
&\quad =P.V.\int_{\R^{N}} \frac{|u(x)-u(x+z)|^{p-2}(u(x)-u(x+z))}{|z|^{N+sp}} dz \\
&\quad =\frac{1}{2} P.V.\int_{\R^{N}} \frac{|u(x)-u(x+z)|^{p-2}(u(x)-u(x+z))+|u(x)-u(x-z)|^{p-2}(u(x)-u(x-z))}{|z|^{N+sp}} dz.
\end{align*}
Let us now show that the $P.V.$ in the above formula can be removed. 
Pick $\e\in (0, 1)$. Then we can write 
\begin{align*}
&\int_{\R^{N}\setminus B_{\e}(0)} \frac{|u(x)-u(x+z)|^{p-2}(u(x)-u(x+z))+|u(x)-u(x-z)|^{p-2}(u(x)-u(x-z))}{|z|^{N+sp}} dz\\
&\quad =\int_{\R^{N}\setminus B_{1}(0)} \frac{|u(x)-u(x+z)|^{p-2}(u(x)-u(x+z))+|u(x)-u(x-z)|^{p-2}(u(x)-u(x-z))}{|z|^{N+sp}} dz\\
&\quad \quad +\int_{B_{1}(0)\setminus B_{\e}(0)} \frac{|u(x)-u(x+z)|^{p-2}(u(x)-u(x+z))+|u(x)-u(x-z)|^{p-2}(u(x)-u(x-z))}{|z|^{N+sp}} dz.
\end{align*}
We begin by proving that
$$
\int_{\R^{N}\setminus B_{1}(0)} \frac{|u(x)-u(x+z)|^{p-2}(u(x)-u(x+z))+|u(x)-u(x-z)|^{p-2}(u(x)-u(x-z))}{|z|^{N+sp}}dz<\infty.
$$
Take $|z|\geq 1$. As $u\in L^{\infty}(\R^{N})$, we see that 
\begin{align*}
\left|\frac{|u(x)-u(x+z)|^{p-2}(u(x)-u(x+z))+|u(x)-u(x-z)|^{p-2}(u(x)-u(x-z))}{|z|^{N+sp}}\right|\leq \frac{2^{p}\|u\|^{p-1}_{L^{\infty}(\R^{N})}}{|z|^{N+sp}},
\end{align*}
which implies the assertion because $\frac{1}{|z|^{N+sp}}$ is integrable at infinity.
Next we demonstrate that
$$
\int_{B_{1}(0)\setminus B_{\e}(0)} \frac{|u(x)-u(x+z)|^{p-2}(u(x)-u(x+z))+|u(x)-u(x-z)|^{p-2}(u(x)-u(x-z))}{|z|^{N+sp}}dz<\infty.
$$
Let $\e\leq |z|<1$. We first assume $p\geq 2$.
Recalling that (see \cite[p. 255]{Simon})
$$
||a|^{p-2}a-|b|^{p-2}b|\leq C_{p}(|a|+|b|)^{p-2}|a-b| \quad \mbox{ for all } a, b\in \R,
$$
and utilizing the mean value theorem, we obtain
\begin{align*}
&\left|\frac{|u(x)-u(x+z)|^{p-2}(u(x)-u(x+z))+|u(x)-u(x-z)|^{p-2}(u(x)-u(x-z))}{|z|^{N+sp}}\right|\\
&\quad =\left|\frac{|\nabla u(x+\theta_{1}z)\cdot z|^{p-2}\nabla u(x+\theta_{1}z)\cdot z-|\nabla u(x-\theta_{2}z)\cdot z|^{p-2}\nabla u(x-\theta_{2}z)\cdot z}{|z|^{N+sp}}\right|\\
&\quad \leq 2^{p-2}C_{p} \|\nabla u\|^{p-2}_{L^{\infty}(\R^{N})}|z|^{p-2} \frac{|[\nabla u(x+\theta_{1}z)-\nabla u(x-\theta_{2}z)]\cdot z|}{|z|^{N+sp}} \\
&\quad \leq 2^{p-2+\gamma}C_{p} \|\nabla u\|^{p-2}_{L^{\infty}(\R^{N})} |z|^{p-2} \frac{[\nabla u]_{C^{0, \gamma}(\R^{N})}|z|^{\gamma+1}}{|z|^{N+sp}} \\
&\quad \leq 2^{p-2+\gamma} C_{p} \|u\|_{C^{1, \gamma}(\R^{N})}^{p-1}\frac{1}{|z|^{N+sp-(\gamma+p-1)}},
\end{align*}
where $\theta_{1}, \theta_{2}\in (0, 1)$ and $[\nabla u]_{C^{0, \gamma}(\R^{N})}=\sup_{\substack{x, y\in \mathbb{R}^{N} \\ x\neq y}}\frac{|\nabla u(x)-\nabla u(y)|}{|x-y|^{\gamma}}$. 
As the function
$$
\frac{1}{|z|^{N+sp-(\gamma+p-1)}}
$$ 
is integrable near the origin (thanks to $\gamma>1-p(1-s)$), we reach the assertion.
Now we suppose $p\in (1, 2)$. Then, observing that (see \cite[p. 255]{Simon})
$$
||a|^{p-2}a-|b|^{p-2}b|\leq C'_{p}|a-b|^{p-1} \quad \mbox{ for all } a, b\in \R,
$$
we get, with the same notations as before,
\begin{align*}
&\left|\frac{|u(x)-u(x+z)|^{p-2}(u(x)-u(x+z))+|u(x)-u(x-z)|^{p-2}(u(x)-u(x-z))}{|z|^{N+sp}}\right|\\
&\quad \leq C'_{p}\frac{| [\nabla u(x+\theta_{1}z)-\nabla u(x-\theta_{2}z)]\cdot z|^{p-1}}{|z|^{N+sp}}\\
&\quad \leq 2^{\gamma(p-1)}C'_{p} [\nabla u]^{p-1}_{C^{0, \gamma}(\R^{N})} |z|^{\gamma(p-1)} \frac{|z|^{p-1}}{|z|^{N+sp}} \\
&\quad \leq 2^{\gamma(p-1)}C'_{p} \|u\|_{C^{1, \gamma}(\R^{N})}^{p-1}\frac{1}{|z|^{N+sp-p+1-\gamma(p-1)}}.
\end{align*}
Due to the fact that 
$$
\frac{1}{|z|^{N+sp-p+1-\gamma(p-1)}}
$$ 
is integrable near the origin (because of $\gamma>\frac{1-p(1-s)}{p-1}$), we can conclude the proof. 
\end{proof}

\begin{remark}
We note that, if $p\in (1, \infty)$ and $s\in (0, \frac{p-1}{p})$, then Lemma \ref{RIF} is valid for all $u\in C^{0, \gamma}(\R^{N})$, with $\gamma\in (0, 1]$ such that $\gamma>\frac{sp}{p-1}$ (indeed, for the estimate near the origin, it suffices to use $||a|^{p-2}a-|b|^{p-2}b|\leq |a|^{p-1}+|b|^{p-1}$ for all $a, b\in \R$, and the H\"older continuity of $u$). From the proof of Lemma \ref{RIF}, we also deduce that $(-\Delta)^{s}_{p}u\in L^{\infty}(\R^{N})\cap C^{0}(\R^{N})$ for all $u\in C^{1, \gamma}(\R^{N})$, with $\gamma\in (0, 1]$ satisfying the restrictions in Lemma \ref{RIF}.
\end{remark}

\noindent
Next we prove a helpful integration by parts formula for $W^{s, p}(\R^{N})\cap C^{1, 1}(\R^{N})$ functions and vector fields of class $C^{0, 1}_{c}(\R^N, \R^N)$.
\begin{lem}\label{LEMMA}
Let $N\in \mathbb{N}$, $s\in (0, 1)$ and $p\in (1, \infty)$. When $p\in (1, 2)$, we assume that $s<\frac{2(p-1)}{p}$. Let $u\in W^{s, p}(\R^{N})\cap C^{1, 1}(\R^{N})$ and $\mathcal{X}\in C^{0, 1}_{c}(\R^{N}, \R^{N})$.  
Then it holds
\begin{align*}
&\frac{C_{N, s, p}}{2} \iint_{\R^{2N}} \frac{|u(x)-u(y)|^{p}}{|x-y|^{N+sp}} \left[\dive(\mathcal{X}(x))+\dive(\mathcal{X}(y))-(N+sp) \frac{(\mathcal{X}(x)-\mathcal{X}(y))\cdot (x-y)}{|x-y|^{2}} \right]  dxdy\\
&=-p \int_{\R^{N}} \mathcal{X}(x)\cdot \nabla u(x) \, (-\Delta)^{s}_{p}u(x) \,dx.
\end{align*}
\end{lem}
\begin{proof}
We first notice that $1>1-p(1-s)$ for all $p\in (1, \infty)$ and $s\in (0, 1)$, whereas $1>\frac{1-p(1-s)}{p-1}$ for all $p\in (1, \infty)$ and $s\in (0, \frac{2(p-1)}{p})$. This means that Lemma \ref{RIF} can be applied to $u$ with $\gamma=1$. In what follows, we modify in a suitable way some arguments found in \cite[Lemma 4.2]{DFW1}. 
Exploiting the symmetry of the kernel $\frac{1}{|x|^{N+sp}}$ and Fubini's theorem, we see that 
\begin{align*}
&\frac{C_{N, s, p}}{2} \iint_{\R^{2N}} \frac{|u(x)-u(y)|^{p}}{|x-y|^{N+sp}} \left[\dive(\mathcal{X}(x))+\dive(\mathcal{X}(y))-(N+sp) \frac{(\mathcal{X}(x)-\mathcal{X}(y))\cdot (x-y)}{|x-y|^{2}} \right]  dxdy \\
&\quad =-\frac{C_{N, s, p}(N+sp)}{2} \iint_{\R^{2N}} |u(x)-u(y)|^{p} \frac{(x-y)\cdot (\mathcal{X}(x)-\mathcal{X}(y))}{|x-y|^{N+sp+2}} dxdy\\
&\quad \quad+\frac{C_{N, s, p}}{2} \iint_{\R^{2N}}  |u(x)-u(y)|^{p} \frac{\dive \mathcal{X}(x)+\dive \mathcal{X}(y)}{|x-y|^{N+sp}}dxdy\\
&\quad =-\frac{C_{N, s, p}(N+sp)}{2} \lim_{\mu\ri 0^{+}}\iint_{|x-y|>\mu} |u(x)-u(y)|^{p} \frac{(x-y)\cdot (\mathcal{X}(x)-\mathcal{X}(y))}{|x-y|^{N+sp+2}} dxdy\\
&\quad \quad+C_{N, s, p} \iint_{\R^{2N}}  |u(x)-u(y)|^{p} \frac{\dive \mathcal{X}(x)}{|x-y|^{N+sp}}dxdy\\
&\quad =-C_{N, s, p}(N+sp) \lim_{\mu\ri 0^{+}} \int_{\R^{N}}\int_{\R^{N}\setminus \overline{B_{\mu}(y)}} |u(x)-u(y)|^{p}  \frac{(x-y)\cdot \mathcal{X}(x)}{|x-y|^{N+sp+2}}dxdy\\
&\quad \quad+C_{N, s, p} \iint_{\R^{2N}}  |u(x)-u(y)|^{p} \frac{\dive \mathcal{X}(x)}{|x-y|^{N+sp}}dxdy.
\end{align*}
Now, applying the divergence theorem in the domain $\R^{N}\setminus \overline{B_{\mu}(y)}$, with $y\in \R^{N}$ and $\mu>0$ fixed, and using the fact that 
$$
\nabla_{x}(|x-y|^{-(N+sp)})=-(N+sp)\frac{x-y}{|x-y|^{N+sp+2}},
$$ 
we have
\begin{align}\label{IBPF1}
&\frac{C_{N, s, p}}{2} \iint_{\R^{2N}} \frac{|u(x)-u(y)|^{p}}{|x-y|^{N+sp}} \left[\dive(\mathcal{X}(x))+\dive(\mathcal{X}(y))-(N+sp) \frac{(\mathcal{X}(x)-\mathcal{X}(y))\cdot (x-y)}{|x-y|^{2}} \right]  dxdy \nonumber\\
&\quad=C_{N, s, p} \lim_{\mu\ri 0^{+}} \int_{\R^{N}}\int_{\R^{N}\setminus \overline{B_{\mu}(y)}} |u(x)-u(y)|^{p} \, \nabla_{x}(|x-y|^{-(N+sp)}) \cdot \mathcal{X}(x) \, dxdy \nonumber\\
&\quad\quad+C_{N, s, p} \iint_{\R^{2N}}  |u(x)-u(y)|^{p} \frac{\dive \mathcal{X}(x)}{|x-y|^{N+sp}}dxdy \nonumber\\
&\quad=-C_{N, s, p} \lim_{\mu\ri 0^{+}}\int_{\R^{N}} \int_{\R^{N}\setminus \overline{B_{\mu}(y)}} |u(x)-u(y)|^{p} \frac{\dive \mathcal{X}(x)}{|x-y|^{N+sp}}dxdy \nonumber\\
&\quad\quad -p C_{N, s, p} \lim_{\mu\ri 0^{+}}\int_{\R^{N}} \int_{\R^{N}\setminus \overline{B_{\mu}(y)}} |u(x)-u(y)|^{p-2} (u(x)-u(y))  \frac{\nabla u(x)\cdot \mathcal{X}(x)}{|x-y|^{N+sp}}dxdy \nonumber\\
&\quad\quad +C_{N, s, p}\lim_{\mu\ri 0^{+}}  \int_{\R^{N}} \int_{\partial B_{\mu}(y)} |u(x)-u(y)|^{p} \frac{(y-x)\cdot \mathcal{X}(x)}{|x-y|^{N+sp+1}} \, d\sigma(y) dx \nonumber\\
&\quad\quad+C_{N, s, p} \iint_{\R^{2N}}  |u(x)-u(y)|^{p} \frac{\dive \mathcal{X}(x)}{|x-y|^{N+sp}}dxdy \nonumber\\
&\quad=-pC_{N, s, p} \lim_{\mu\ri 0^{+}} \iint_{|x-y|>\mu} |u(x)-u(y)|^{p-2}(u(x)-u(y)) \frac{\nabla u(x)\cdot \mathcal{X}(x)}{|x-y|^{N+sp}} dxdy\nonumber\\
&\quad\quad+C_{N, s, p} \lim_{\mu\ri 0^{+}}  \mu^{-N-1-sp} \iint_{|x-y|=\mu} |u(x)-u(y)|^{p} (y-x)\cdot \mathcal{X}(x) d\sigma(x, y) \nonumber\\
&\quad=-\frac{pC_{N, s, p}}{2} \lim_{\mu\ri 0^{+}} \int_{\R^{N}} \nabla u(x)\cdot \mathcal{X}(x) \int_{\R^{N}\setminus \overline{B_{\mu}(0)}} \left( \frac{|u(x)-u(x+z)|^{p-2}(u(x)-u(x+z))}{|z|^{N+sp}} \right. \nonumber\\
&\quad\quad \left.+\frac{|u(x)-u(x-z)|^{p-2}(u(x)-u(x-z))}{|z|^{N+sp}}\right) dzdx\nonumber\\
&\quad\quad+\frac{C_{N, s, p}}{2} \lim_{\mu\ri 0^{+}}  \mu^{-N-1-sp} \iint_{|x-y|=\mu} |u(x)-u(y)|^{p} (y-x)\cdot (\mathcal{X}(x)-\mathcal{X}(y)) d\sigma(x, y)\nonumber\\
&\quad=-p Y_{1}+Y_{2}.
\end{align} 
It follows from Lemma \ref{RIF} that
\begin{align}\label{IBPF2}
Y_{1}&=\frac{C_{N, s, p}}{2}  \int_{\R^{N}}  \nabla u(x)\cdot \mathcal{X}(x) \int_{\R^{N}} \left(\frac{|u(x)-u(x+z)|^{p-2}(u(x)-u(x+z))}{|z|^{N+sp}}\right. \nonumber\\
&\quad \left. +\frac{|u(x)-u(x-z)|^{p-2}(u(x)-u(x-z))}{|z|^{N+sp}} \right) dzdx \nonumber\\
&=\int_{\R^{N}} \nabla u(x)\cdot \mathcal{X}(x)  \, (-\Delta)^{s}_{p}u(x) \,dx.
\end{align}
Next we claim that
\begin{align}\label{IBPF3}
Y_{2}=0.
\end{align}
Indeed, as $\mathcal{X}$ has compact support, we can find $R>0$ large enough such that $(\mathcal{X}(x)-\mathcal{X}(y))=0$ for all $x, y\in B_{R}(0)$ with $|x-y|<1$. For all $\mu\in (0, 1)$, we put
$$
\mathcal{N}_{\mu}=\Bigl\{(x, y)\in B_{R}(0)\times B_{R}(0): |x-y|=\mu \Bigr\}.
$$
Because $u\in C^{0, 1}(\R^{N})$ and $\mathcal{X}\in C^{0, 1}(\R^{N}, \R^{N})$, and using that the $2N-1$-dimensional measure of the set $\mathcal{N}_{\mu}$ is of order $O(N-1)$ as $\mu\ri 0^{+}$, we have
\begin{align*}
&\mu^{-N-1-sp} \iint_{|x-y|=\mu} |u(x)-u(y)|^{p}(y-x)\cdot (\mathcal{X}(x)-\mathcal{X}(y)) d\sigma(x, y)\\
&\quad=\mu^{-N-1-sp} \iint_{\mathcal{N}_{\mu}} |u(x)-u(y)|^{p}(y-x)\cdot (\mathcal{X}(x)-\mathcal{X}(y)) d\sigma(x, y)=O(\mu^{p(1-s)})\ri 0 \quad \mbox{ as } \mu\ri 0^{+},
\end{align*}
that is, \eqref{IBPF3} is valid. Combining \eqref{IBPF1}, \eqref{IBPF2}, and \eqref{IBPF3}, we obtain the desired formula.
\end{proof}

\noindent
Now, we are ready to provide the proof of Theorem \ref{thm2}.
\begin{proof}[Proof of Theorem \ref{thm2}]
Let $u\in W^{s, p}(\R^{N})\cap C^{1, 1}(\R^{N})$ be such that $(-\Delta)^{s}_{p}u=g(u)$ in $\R^{N}$. Pick $\varphi\in C^{1}_{c}(\R^{N})$ such that $0\leq \varphi\leq 1$ in $\R^{N}$, $\varphi(x)=1$ for $|x|\leq 1$ and $\varphi(x)=0$ for $|x|\geq 2$. Put $\varphi_{\lambda}(x)=\varphi(\lambda x)$ for all $x\in \R^{N}$ and $\lambda>0$.
Note that, for all $x\in \R^{N}$ and $\lambda>0$, 
\begin{align}\label{WS}
0\leq\varphi_{\lambda}(x)\leq 1 \,\mbox{ and } \, |x||\nabla \varphi_{\lambda}(x)|\leq C_{1},
\end{align}
 where the constant $C_{1}>0$ is independent of $\lambda$. 
As the $C^1$ vector field $\mathcal{X}_{\lambda}(x)=\varphi_{\lambda}(x)x\in C^{0, 1}_{c}(\R^{N}, \R^{N})$ (note that \eqref{WS} implies that $|x\varphi_{\lambda}(x)-y\varphi_{\lambda}(y)|\leq C_{2}|x-y|$ for some constant $C_{2}>0$ independent of $\lambda$), and $u\in W^{s, p}(\R^{N})\cap C^{1, 1}(\R^{N})$, we can apply 
Lemma \ref{LEMMA} to deduce that 
\begin{align*}
&\frac{C_{N, s, p}}{2p} \iint_{\R^{2N}} \frac{|u(x)-u(y)|^{p}}{|x-y|^{N+sp}} \left[\dive(\mathcal{X}_{\lambda}(x))+\dive(\mathcal{X}_{\lambda}(y))-(N+sp) \frac{(\mathcal{X}_{\lambda}(x)-\mathcal{X}_{\lambda}(y))\cdot (x-y)}{|x-y|^{2}} \right]  dxdy\\
&=-\int_{\R^{N}} \mathcal{X}_{\lambda}(x)\cdot \nabla u(x) \, (-\Delta)^{s}_{p}u(x) \,dx\\
&=-\int_{\R^{N}} \mathcal{X}_{\lambda}(x)\cdot \nabla u(x) \, g(u(x)) \,dx.
\end{align*}
An integration by parts ensures that
\begin{align*}
-\int_{\R^{N}} g(u(x)) \varphi_{\lambda}(x)x\cdot \nabla u(x)dx=N \int_{\R^{N}} \varphi_{\lambda}(x)G(u(x))dx+\lambda \int_{\R^{N}} G(u(x)) x\cdot \nabla \varphi_{\lambda}(x)dx,
\end{align*}
and so
\begin{align}\label{BRUN1}
&\frac{C_{N, s, p}}{2p} \iint_{\R^{2N}} \frac{|u(x)-u(y)|^{p}}{|x-y|^{N+sp}} \left[\dive(\mathcal{X}_{\lambda}(x))+\dive(\mathcal{X}_{\lambda}(y))-(N+sp) \frac{(\mathcal{X}_{\lambda}(x)-\mathcal{X}_{\lambda}(y))\cdot (x-y)}{|x-y|^{2}} \right]  dxdy \nonumber\\
&=N \int_{\R^{N}} \varphi_{\lambda}(x)G(u(x))dx+\lambda \int_{\R^{N}} G(u(x)) x\cdot \nabla \varphi_{\lambda}(x)dx.
\end{align}
From \eqref{WS}, we see that, for all $x, y\in \R^{N}$ with $x\neq y$ and $\lambda>0$, 
$$
\left| \dive(\mathcal{X}_{\lambda}(x))+\dive(\mathcal{X}_{\lambda}(y))-(N+sp) \frac{(\mathcal{X}_{\lambda}(x)-\mathcal{X}_{\lambda}(y))\cdot (x-y)}{|x-y|^{2}}\right|\leq C_{3},
$$
for some constant $C_{3}>0$ independent of $\lambda$. In view of this fact, $u\in W^{s, p}(\R^{N})$, \eqref{WS}, the pointwise convergences $\varphi_{\lambda}(x)\ri 1$, $\mathcal{X}_{\lambda}(x)\ri x$ and $\dive\mathcal{X}_{\lambda}(x)\ri N$ for all $x\in \R^N$, as $\lambda\ri 0^{+}$, and $G(u)\in L^{1}(\R^{N})$, we can utilize the dominated convergence theorem to infer that, as $\lambda\ri 0^{+}$,
\begin{align}\label{BRUN2}
&\frac{C_{N, s, p}}{2p} \iint_{\R^{2N}} \frac{|u(x)-u(y)|^{p}}{|x-y|^{N+sp}} \left[\dive(\mathcal{X}_{\lambda}(x))+\dive(\mathcal{X}_{\lambda}(y))-(N+sp) \frac{(\mathcal{X}_{\lambda}(x)-\mathcal{X}_{\lambda}(y))\cdot (x-y)}{|x-y|^{2}} \right]  dxdy \nonumber\\
&\ri \frac{C_{N, s, p}(N-sp)}{2p} \iint_{\R^{2N}} \frac{|u(x)-u(y)|^{p}}{|x-y|^{N+sp}} dxdy,
\end{align}
and
\begin{align}\label{BRUN3}
N \int_{\R^{N}} \varphi_{\lambda}(x)G(u(x))dx+\lambda \int_{\R^{N}} G(u(x)) x\cdot \nabla \varphi_{\lambda}(x)dx\ri N \int_{\R^{N}} G(u(x))dx.
\end{align}
On account of \eqref{BRUN1}, \eqref{BRUN2}, and \eqref{BRUN3}, we arrive at \eqref{POHOZAEV}. The proof is now complete.
\end{proof}

\smallskip

\noindent
\section*{Acknowledgements} 
The author warmly thanks the anonymous referees for their valuable comments on the paper.

\end{document}